\documentclass[12pt]{article}
\usepackage{times,amssymb,amsmath,amsfonts}   
\usepackage{ifthen} 
\begin{document}

\title{A Poincar\'e Inequality on Loop Spaces}
\author{Xin Chen,  Xue-Mei Li and Bo Wu\\
Mathemtics Institute\\
University of Warwick\\
Coventry CV4 7AL, U.K.}
\date{\today}
\maketitle

\newcommand{\A}{{\bf \mathcal A}}
\def\a{{{\alpha }}}
\newcommand{\B}{{ \bf \mathcal B }}
\newcommand{\C}{{\mathcal C}}
\def\capa{{\mathop {Cap}}}
\def\cov{{\mathop {\rm \Gamma}}}
\newcommand{\D}{{\rm I \! D}}
\newcommand{\E}{{\bf E}}
\def\Cyl{{\mathop {Cyl}}}
\def\Ent{{{\bf Ent}}}
\newcommand{\F}{{\mathcal F}}
\newcommand{\G}{{\mathcal G}}
\newcommand{\K}{{\mathcal K}}
\newcommand{\p}{{\bf P}}
\newcommand{\R}{{\bf  R}}
\newcommand{\HH}{{\mathcal H}}
\newcommand{\I}{{\mathcal I}}
\newcommand{\norm}{{ \rule{0.2ex}{2ex} }}
\def\ent{{\mathop {\bf Ent}}}
\def\Lo{{\mathcal L}}
\def\Ric{{\mathop {Ric}}}
\def\Hess{{\mathop {Hess}}}
\def\n#1{|\kern-0.24em|\kern-0.24em|#1|\kern-0.24em|\kern-0.24em|}
\def\paral{/\kern-0.3em/}
\def\parals_#1{/\kern-0.3em/_{\!#1}}
\def\1{{\mathop {\bf 1}}}
\def\ie{{\rm i.e.}}
\def\var{{{\rm \bf Var}}}
\def\osc{{{\bf Osc}}}
\renewcommand{\thefootnote}{}
\makeatletter
\def\sqr#1#2{%
	\fboxsep0pt%
	\fboxrule#2pt%
	\fbox{\hbox to #1pt{\hss\vbox to #1pt{\vss}}}}
\def\qed{\unskip\nobreak\hfil\penalty50\hskip2em\hbox{}\nobreak
      \hfil\sqr{5.5}{0.2}\parfillskip=\z@ \finalhyphendemerits=0 \par}
\newenvironment{proof}[1][ ]{\ifthenelse{\equal{#1}{ }}{\def\MH@pfName{.}}{\def\MH@pfName{ #1.}}%
	\trivlist\item[\hskip\labelsep{\it Proof\MH@pfName}]\ignorespaces}{\phantom{a}\nobreak\qed\endtrivlist}
\makeatother

\newtheorem{theorem}{Theorem}[section] \newtheorem{proposition}[theorem]{%
Proposition} \newtheorem{lemma}[theorem]{Lemma}
 \newtheorem{corollary}[theorem]{Corollary} \newtheorem{definition}{Definition}[section]
\newtheorem{remark}[theorem]{Remark}
\newtheorem{assumption}[theorem]{Assumption}

\begin{abstract}
We investigate properties of measures in infinite dimensional spaces in terms of Poincar\'e inequalities. 
 A Poincar\'e inequality states that the $L^2$ variance of an admissible  function is controlled by the  homogeneous $H^1$ norm.  In the case of Loop spaces, it was observed by L. Gross \cite{Gross91}  that the homogeneous $H^1$ norm alone may not control the $L^2$ norm and  a potential term involving the end value of the Brownian bridge is introduced.  Aida, on the other hand, introduced a weight on the Dirichlet form.  We show that Aida's modified Logarithmic Sobolev inequality implies weak Logarithmic Sobolev Inequalities and weak Poincar\'e inequalities with precise estimates on the order of convergence. The order of convergence in the weak Sobolev inequalities are related to weak $L^1$ estimates on the weight function. This and a relation between Logarithmic Sobolev inequalities and weak Poincar\'e inequalities lead  to a Poincar\'e inequality on the loop space over certain manifolds.
 
 \end{abstract}

\section{Introduction}
A Poincar\'e inequality is of the form
$$\int_N (f-\bar f)^2 \mu(dx) \leqslant {1\over C}\int_N |\nabla f|^2 \mu(dx),$$
where $f$ ranges through an admissible set of real valued functions on a space $N$, $\nabla$ is a gradient type operator,  
$\mu$ a finite measure on $N$ and hence is often normalised to have total mass $1$, and $\bar f=\int f d\mu$.  For $N=[0,L]$, $\mu$ the normalised Lebesgue measure, the constant $C$ is $4\pi^2\over L^2$  for $C^1$ functions satisfying the Dirichlet boundary or  the periodic boundary conditions. More generally if  $N$ is a compact closed Riemannian manifold, $dx$ the volume measure and $\nabla$ the Riemannian gradient operator, the best constant  in the Poincar\'e inequality is given by taking infimum of the Raleigh quotient 
$$\int_N |d f|^2 dx \over \int_N f^2 dx$$
over the set of non-constant  smooth functions of zero mean. For this reason Poincar\'e inequality is associated with  the study of the spectral properties of  the Laplacian operator and hence the underlying  Riemannian geometry.  For quasi isometric Riemannian manifolds, if a Poincar\'e inequality holds for one manifold it holds for the other.

The Poincar\'e constant $C=\lambda_1$,  that is the first non-trivial eigenvalue of the Laplacian on a compact manifold,  is  related to the isoperimetric constant in Cheeger's isoperimetric inequality. Standard isoperimetric inequalities say that  for an open bounded set $A$ in $\R^n$, the ratio between the area of its boundary $\partial A$ and the volume of $A$  to the power of $1-{1\over n}$ is minimised by the unit ball.  
  In $\R^2$, it means that $L^2 \geqslant  4 \pi A$ where $A$ and $L$ are respectively the area of an open set  and $L$ the length of its boundary.  By the Federer-Fleming theorem the isoperimetric constant is the same as
  $\inf_{f \in C_K^\infty}  {  \|\nabla f\|_{L_1}\over  \|f\|_{{n\over n-1}}}$.

   In relation to Poincar\'e inequality, especially in infinite dimensions,  the more useful form of  isoperimetric inequality is that of Cheeger. Following Cheeger let
$$h=\inf_{A} {\mu(\partial A)\over \min\{\mu(A), \mu(M/A)\}}.$$
where the infimum is taken over all open subsets of $M$. Then $h^2\leqslant 4\lambda_1$ by Cheeger \cite{Cheeger-inequality}. On the other hand let $K$ be  the lower bound of the Ricci curvature.  
Then it is shown by Buser  \cite{Buser82} that $\lambda_1\leqslant C(\sqrt K h+h^2)$ for which M. Ledoux \cite{Ledoux-inequality-94} has a beautiful analytic proof.  Versions of isoperimetric inequalities
for Gaussian measures  in infinite dimensional spaces are explained in Ledoux \cite{ Ledoux-StFlour} and Ledoux-Talagrand \cite{Ledoux-Talagrand}. 
\medskip

We take the view that the Poincar\'e inequality describes properties of the measure $\mu$ for a given gradient operator.  Poincar\'e inequality does not hold for $\R^n$ with Lebesgue measure. It does hold for the Gaussian measure. For the standard normalised Gaussian measure, the Poincar\'e constant is $1$
and the corresponding eigenfunction of the Laplacian is the Hermitian polynomial $x/2$. 
If  $h$ is a smooth function $\mu$ a measure which is absolutely continuous with respect to the Lebesgue measure with density $e^{-2h}$, for any $f$ in the domain of $d$,
$$\int_N |df|^2(x) \mu(dx)=-\int_N \langle f, \Delta f\rangle (x)\mu(dx)-2\int_N\langle df, dh\rangle \mu(dx).$$
The corresponding Poincar\'e inequality is then related to the Raleigh quotient  of the Bismut-Witten Laplacian 
$\Delta^h:=\Delta+2 L_{\nabla h}$ on $L^2(M, e^{-2h}dx)$. The  Bismut-Witten Laplacian
$$\Delta^h: L^2(M, e^{-2h}dx)\to L^2(M, e^{-2h}dx)$$ is unitarily equivalent to the following  linear operator on  $L^2(M, dx)$:
$$\square^h=\Delta+(|dh|^2+\Delta h).$$
The spectral property of $\Delta^h$, hence the validity of the Poincar\'e inequality for $\mu$ is determined by the spectral property of the Schr\"odinger operator $\square^h$ on $L^2(M; dx)$. 

\bigskip

{\bf The state space.} 
A number of infinite dimensional spaces have been the objects of study. They include the
space of paths over a finite state space, in particular the space of  loops, or more generally space of maps. 
Our interest  in path spaces comes from the desire to understand regularity properties of measures
which are distributions of important  stochastic processes and to establish a related Sobolev calculus.
By path space we mean the space of continuous paths which are not necessarily smooth, of which Wiener space $\Omega$ with Wiener measure $\p$ is a primary example. 
Other natural measures are those induced by stochastic processes  such as the Brownian Bridge measure.  The properties of Brownian Bridge measures are non-trivial. They are singular measures with respect to  the Wiener measure.  For the Wiener space the gradient operator would be that related to the Cameron-Martin space of the measure.   Interesting functions on the Wiener space such as stochastic integrals are not in general differentiable as real valued functions on the Banach space $\Omega$. They are on the other hand often differentiable in the sense of Malliavin calculus where the functions are differentiated in the directions of the Cameron-Martin space, also called H-differentiation. This will play the role of 
 the standard differentiation on a differentiable manifold. The corresponding gradient operator will be used in the formulation of Poincar\'e inequality with respect to measures on the Wiener space and  on more general spaces of continuous paths. 
\medskip

{\bf Main Results. }Although a Logarithmic Sobolev inequality holds for the Brownian bridge measure on the Wiener space and for the Brownian motion measure on the path space over a compact manifolds, it may not hold on  a general loop space.  As noted by L. Gross, \cite{Gross91}, Poincar\'e inequalities 
 do not hold on the Lie group $S^1$ due to the lack of connectedness of the loop space.   A. Eberle, \cite{Eberle-absence-gap}, gave an example of a compact simply connected Riemannian manifold on which the Poincar\'e inequality does not hold for the Brownian bridge measure. Driver-Lohrenz \cite{Driver-Lohrenz96} showed that Logarithmic Sobolev inequalities hold on loop groups for the heat kernel measure on loop spaces  over a compact type Lie group. For the Brownian bridge measure a positive result was obtained by Aida for the Hyperbolic space $H$ where he obtained a weak form of Logarithmic Sobolev inequality with a weight function. We show here that Aida's type weak logarithmic Sobolev inequality leads to a weak logarithmic Sobolev inequality using the non-homogeneous $H_1$ norm together with an $L^\infty$ norm. We also show that there is a precise passage from weak Logarithmic Sobolev inequality  to weak Poincar\'e inequality. As a corollary we obtain a Poincar\'e inequality for the Brownian bridge measure on loop spaces over the hyperbolic space where the Bismut tangent space  is defined using the Levi-Civita connection. 
 
 \medskip
 
 {\bf Acknowledgement.}
 We would like to thank Martin Hairer for stimulating discussions and for pointing to look into the work of Guillin et al. This research is supported by  the EPSRC( EP/E058124/1).
 
\section{The Missing Arguments}

 On a compact manifold,  Poincar\'e inequality for the Laplace-Beltrami
operator is proved by showing that
$$\inf_{f\in H^1, |f|_{L^2}=1, \int f =0} \int_M |\nabla f|^2 dx$$ is attained, by a non constant function.
The main ingredient for this method to work is the Rellich-Kondrachov compact embedding  theorem of $H^{1,q}$ into $L^p$, which we do not have
 in the infinite dimensional situation. 
The other approach is the dynamic one which we will now explain. It is equivalent to consider the corresponding operator on differential 1-forms. By a Riemannian manifold we mean a connected Riemannian manifold.

We give the standard semi-group argument which in principle works for measures on infinite dimensional spaces. For better understanding assume that the measure concerned is on a  finite dimensional Riemannian manifold. Let $M$ be a smooth complete manifold and for $x_0\in M$ let $(F_t(x_0,\omega), t  \geqslant  0)$ be the solution flow to a stochastic differential equation
$$dx_t=\sum_{i=1}^m X_i(x_t)\circ dB_t^i +X_0(x_t)dt$$
with initial value $x_0$. Here $X_i$ are smooth vector fields and $\omega$ the chance variable. Let $\mu_t$ be the law of $F_t$ with initial distribution $\mu_0$. It is given by
$$\mu_t(A)=\int_{x \in M} P(F_t(x) \in A)\mu_0(dx).$$  
 If the system is elliptic the $X_i's$ induces a Riemannian metric and the infinitesimal generator is of the form ${1\over 2} \Delta+A$ for $\Delta$ the Laplace-Beltrami operator for the corresponding Levi-Civita connection and $A$ a vector field called the drift. Suppose that the drift  is of gradient form given by a potential function $h$. Then the system has an invariant measure $\mu(dx)=e^{2h}dx$ which is finite for example if $\Ric_x-2\Hess_x (h)>\rho$ for a positive number $\rho$. Here $\Ric$ denotes the Ricci curvature for the intrinsic Riemannian metric. More generally the finiteness of the invariant measure holds even if the lower bound $\rho$ depends on $x$ provided that the quantity
 $$\sup_{x\in K} \int_0^\infty \E e^{-\int_0^t \rho(F_s(x, \omega)) ds}dt,$$
  is finite for any given compact subset $K$,
 see \cite{group} \cite{flow}.  In the following we assume that the system has an finite invariant measure $\mu$ and we assume that $P_tf$ converges in $L^2(M;\mu)$ as $t$ goes to infinity.  Then
 \begin{eqnarray*}
\int_M (f -\bar f)^2d\mu&=&\int_M \left( f^2 - \bar f^2 \right)\; d\mu=  \lim_{t\to \infty} \int_M(f^2 -(P_t  f)^2)(x)d\mu(x)\\
&=&-  \lim_{t\to \infty} \int_M  \int_0^t  {\partial \over \partial s}(P_s f)^2 ds \; d\mu\\
&=&\lim_{t\to \infty}\int_0^t \int_M  (dP_s f )^2 \;d\mu \; ds 
=\int_0^\infty \int_M (dP_sf)^2\; d\mu\; ds .
\end{eqnarray*} 
Here $d^*$ is the $L^2$ adjoint  of the differential operator $d$ with respect to the measure $\mu$.  For $v_0\in T_{x_0}M$, let $TF_t(\omega)(v_0)$ be the spatial derivative of $F_t(x_0, \omega_0)$ in the direction of $v_0$ which in general only exists in the $L^2$ sense.  Define
$$\delta P_t(df)(v_0)=\E df(TF_t(\omega)(v_0)).$$
This extends to  a semi-group on bounded differential 1-forms and under suitable conditions solves a corresponding partial differential equation on differential 1-forms.
Assume  that $d(P_tf)=\delta P_t(df)$, see \cite{flow} for conditions for this to hold.  The condition $\Ric_x-2\Hess_x (h)>\rho$ for some constant $\rho>0$ implies that the norm of the conditional expectation of the derivative flow is controlled by  $  e^{-\rho t}$, see \cite{moment-stability} for more precise estimate, and hence we have control for $d(P_tf)$ and
  \begin{eqnarray*}
\int_M (f -\bar f)^2d\mu&\le&
\int_0^\infty \int_M \E |df|^2(F_t((x, \omega)) \; d\mu\;  e^{-\rho s} ds \\
 &=& {1\over  \rho}  \int_M  |df|^2(x) \; d\mu.
\end{eqnarray*} 
 
This proof using the equivalence of  Poincar\'e inequality  and the semi-group inequality  $|P_t|^2_{L^2}\leqslant e^{- \rho  t}$.
The condition $\Ric- 2\Hess(h)$ is bounded from below
is called Bakry-Emery condition \cite{Bakry-Emery}.  In the case of $M=\R^n$, the standard Gaussian measure corresponding to a system with $\Ric\equiv 0$ and  the Bakry-Emery condition is exactly the log-convexity condition on measures.  In this case $h(x)=-{x^2\over 4}$ and the constant  in the Poincar\'e inequality is $1$. The Poincar\'e theorem  above can be considered as a generalisation to the Lichnerowicz Theorem, a  standard theorem in  Riemannian geometry which gives a lower bound for the  first eigenvalue of the Laplacian in  terms of  the lower bound on the Ricci curvature. 

In fact under the assumptions given above the stronger Logarithmic Sobolev inequality holds:
$$\int f^2 \log {f^2\over \E|f|^2} \mu(dx)\leqslant {2\over \rho}\int |\nabla f|^2 \mu(dx).$$
For the standard Gaussian measure the logarithmic Sobolev constant is $2$. The proof is virtually the same. We apply the same argument to the function  $P_t f\log P_tf$, with limit $\bar f  \log(\bar f)$, instead of to $f$ on functions bounded below by a positive constant. A Fatou lemma allows the extension to positive functions. The final result is obtained by applying the same argument  to $| f|$ and observe that $|\nabla |f| | =|\nabla f|$.

Instead of  the equilibrium measure $\mu$ on the finite dimensional Riemannian manifold, we study the law of a stochastic process $(F_t(\omega), 0\leqslant t\leqslant T)$ on the space of paths over $M$, of which the Wiener measure on the Wiener space is a special case.  To apply the semi-group argument we would need to have
a good understanding of the semi-group associated to $d^*d$ and corresponding semi-groups on differential 1-forms which is itself an issue to be resolved, except in the case of the classical Wiener space. The semi-group argument is modified and the standard method is the Clark-Ocone formula approach, which combines the problem of defining the unbounded operator $d$ with the investigation of the measure itself.

 \subsection{ Poincar\'e Inequality for Gaussian Measures}
First let $\mu$ be a a Gaussian measure  whose support is a finite dimensional vector space, $\R^n$. It is not surprising that a function $f$  differentiable in $\R^n$ with  $d f=0$ is a constant on this subspace. Let $B$ be a Banach space and $\mu$ a mean zero Gaussian measure  with $B$  its topological support  and covariance operator $\cov$.  The Cameron-Martin space $H$ is the intersection of all vector subspaces of $B$ of full measure and it is a dense set of $B$.  Yet  the Gaussian measure $\mu$ does not charge $H$, $\mu(H)=0$.  And $\mu$ is quasi translation invariant precisely in the directions of vectors of $H$.  Let  $f: B\to \R$ be an $L^2$ function differentiable in the directions of $H$ and let  $\nabla f\equiv \nabla_Hf$, an element of $H$,  be the gradient of $f$ defined by
$\langle \nabla f, h\rangle_H=df(h)$. The square of the $H$-norm of the gradient $f$  is precisely $\sum_i |df(h_i)|^2$ where $h_i$ is an orthonormal basis of $H$. There is a corresponding quadratic form:
$\int _B |\nabla f|^2_H(x) \mu(dx).$ 

 When $B$ is a Hilbert space the Cameron-Martin space is the range of $\cov^{1\over 2}$ and $\cov$ can be considered as a  trace class linear operator on $B$.
 If $f$ is a $BC^1$ function, $\nabla_Bf$ is defined and $\nabla_Hf=\cov  \nabla_Bf$.
The associated quadratic form is $\int_B |\cov^{-1/2} \nabla_Bf|_B^2  d\mu(x)$
and  the Poincar\'e inequality becomes, for $f$ with zero mean,
$$\int f^2(x)\mu(dx) \leqslant {1\over C}\int_B |\cov^{-1/2} \nabla_Hf|_B^2  d\mu(x).$$
 To the quadratic form 
$\int_B |\cov^{-{1\over 2}} \nabla_Hf|_B^2  d\mu(x)$ there associates a linear operator
$\Lo$ given by
$$\int f\Lo g d\mu=\int \langle \nabla_Hf, \Gamma^{-1} \nabla_Hg\rangle_B \;d\mu.$$
The dynamic of the corresponding semi-group is given by  the solution of the the Langevin equation  $du_t=dW_t-{1\over 2}u_tdt$, where  $W_t$ is a cylindrical Wiener process on $H$.

\bigskip

For $T$ any given positive number, define $$ C_0(\R^m)\equiv \Omega=\{ \sigma: [0,T] \to \R^m:   \sigma(0)=0 \hbox{ continuous} \}.$$
The standard Wiener measure $\p$  on $\Omega$ is a Gaussian measure with Covariance
$$\cov(l_1,l_2)=\int_0^T\int_0^T (s\wedge t) d\mu_{\ell_1}(s)d \mu_{\ell_2}(t)$$
where $\mu_{\ell_i}$ are measures on $[0,T]$ associated to $\ell_i\in \Omega^*$. Its associated Cameron-Martin space is the Sobolev space on $\R^n$ consisting of paths in $\Omega$ with finite energy
$$H=\left\{h: [0,T]\to \R^m \hbox{ such that } \int_0^T |\dot h_t|^2 dt<\infty\right\}.$$
Denote by $C_K^\infty$ the space of real valued functions on $N$ with compact support. Let $$ {\rm Cyl}=\{f(\omega_{t_1}, \dots, \omega_{t_k}),  f\in C_K^\infty(\stackrel{k}{\overbrace{\R^m\times\dots \times  \R^m}}), 0<t_1\leqslant \dots \leqslant t_k\leqslant T\}.$$
For the cylindrical function $f$, 
$$df(\omega)(h)=\sum_{i=1}^k \partial_i f(\omega_{t_1}, \dots, \omega_{t_k})(h_{t_i}),$$
where $\partial_i f$ stands for differentiation with respect to i-th variable.
Hence  $$\nabla f(\omega)(t)=\sum_{i=1}^k \partial_i f (\omega_{t_1}, \dots, \omega_{t_k})  t\wedge t_i$$
where $t\wedge t_i$ denotes $\min(t,t_i)$.
The gradient operator, more precisely the associated quadratic form, is associated to the Laplace operator
$\Lo=-{1\over 2}d^*d$, where $d^*: L^2( {\Bbb L}(\Omega \to H), \p)\to L^2(\Omega, \p)$ is the adjoint of the
differential operator $d$. Note that $d^*$ depends on the measure $\mu$ and the norm on the Cameron-Martin space. It is also called the number operator as it acts as a multiplication operator on each chaos of the Wiener Chaos decomposition of the $L^2$ space: $L^2(\Omega,\mu)=\oplus_{k=0}^\infty H_k$. Then $d^*d f=\sum_{n=0}^\infty k I_k(f)$, where $I_k(f)$ is the orthogonal projection of $f$ to the $k$-th chaos $H_k$. The operator $d$ whose initial domain the set of smooth cylindrical functions with compact support is known to be a closable operator.  Let $\D^{1,2}$ be the closure of $d$ under the graph norm with the graph norm
$|f|^2_{L^2}+\int |\nabla f|^2 \, d\mu$.  These are referred as the Sobolev space (defined by H-differentiation).
 The  Gaussian Sobolev space structure can be given to any mean zero Gaussian measures and a Poincar\'e inequality related to the gradient can be shown to be valid for all functions in $\D^{2,1}$ with Poincar\'e constant $1$. The classical approach to this is to use the symmetric property, rotation invariance, of the Gaussian measure.  It is Gross, \cite{Gross75}, who obtained the Logarithmic Sobolev inequality and notices its validity in an infinite dimensional space and its relation with Nelson's hypercontractivity.  A number of simple proofs  have since been given. The dynamic argument we outlined earlier also works as the Ornstein-Uhlenbeck semi-group $P_t$ for $\Lo$  has the commutation property: $\nabla P_t f=e^{-t}P_t(\nabla f)$. 
\bigskip

  The Brownian Bridge measure  $\nu_{0,0}$ is the law of the Brownian bridge starting and ending at $0$, one of whose realisation is $B_t -{t\over T}B_T$. It can also be realised as solution to the time-inhomogeneous SDE $dx_t=dB_t-{x_t\over T-t}dt$.
The Brownian bridge measure is a Radon Gaussian measure and Gaussian measure theory applies to give the required Logarithmic Sobolev inequality as well as the Poincar\'e inequality with Poincar\'e constant $1$.

\subsection{The Path Spaces}

Let $M$ be a smooth finite dimensional Riemannian manifold which is stochastically complete.
A Brownian motion on $M$ is the strong Markov process $x_t$ with values in $M$ such that probability density 
of $x_t$ is the heat kernel $p_t(x,y)$. By stochastically complete we mean that $\int_M p_t(x,y)dy=1$, which holds true if the lower bound of the Ricci curvature , $\Ric_x=\inf_{|v|=1} Ric_x(v,v)$,
goes to minus infinity slower than $-d^2(x)$, where  $d(x)$ denotes the Riemannian distance of $x$ from
a fixed point $x_0\in M$, or by a result of Grigor'yan \cite{Grigoryan}  if the growth of the volume of  geodesic balls of radius $r$ has an upper bound of the type: $\int^\infty {r dr \over \log vol(B_{x_0}(r))}=\infty$.  Fix a number $T>0$. We define the path space on $M$ based at $x_0$ as 
 $$\C_{x_0}M=\{\sigma: [0,T]\to M, \sigma(0)=x_0 | \hbox{ $\sigma$ is continuous}\}.$$
It is  Banach manifold modelled on the Wiener space $\C_0\R^n$ for $n$ the dimension of the manifold. It is also a complete separable metric space with distance function $\rho$ given by:
 $$\rho(\sigma_1, \sigma_2)=\sup_t d(\sigma_1(t), \sigma_2(t)\}.$$
 For $y_0\in M$, define 
\begin{eqnarray*}
\C_{x_0, y_0}M&=&\{ \sigma \in \C_{x_0}M \;| \;\sigma(T)=y_0\}\\
L_{x_0}M&=&\{ \sigma \in \C_{x_0}M\,|\,\sigma(T)=x_0\}.
\end{eqnarray*}
Both $\C_{x_0, y_0}M$ and $L_{x_0}M$ are closed subspaces of $\C_{x_0}M$ viewed as a metric space.

\bigskip

The Brownian motion measure $\mu_{x_0}$ on $\C_{x_0}M$  is  the pushed forward measure of $\p$ by the Brownian motion.  We view the Brownian motion measure dynamically.  Define the space of cylindrical functions:
$$  \Cyl=\{F| F(\sigma)=f(\sigma_{t_1}, \dots, \sigma_{t_k}), f\in C^\infty_K (M^k),  t_0<t_1<\dots <t_k\leqslant T\}.$$
Then 
\begin{eqnarray*}
&&\int_{\C_{x_0}M} f(\sigma_{t_1}, \dots, \sigma_{t_k})d\mu_{x_0}(\sigma)\\
&&=\int_{M}\dots \int_{M} f(x_1,\dots, x_k)p_{t_1}(x_0, x_1)p_{t_2-t_1}(x_1,x_2)\dots p_{t_k-t_{k-1}}(x_{k-1}, x_k)\Pi_idx_i.
\end{eqnarray*}

Let $ev_t: \C_{x_0}M\to \R$ be the evaluation map at time $t$. The conditional law of the canonical process
$(ev_t, t\in [0,T])$  on $\C_{x_0}M$ given $ev_T(\sigma)=y_0$ is denoted by $\mu_{x_0,y_0}$, hence for a Borel set $A$ of
 $\C_{x_0}M$,
\begin{equation}
\mu_{x_0,y_0}(A)=\mu_{x_0} (\sigma \in A| \sigma_T=y_0).
\end{equation}
Restricted to $\F_t$ for $t<T$ the two measures are absolutely continuous with respect to each other  with Radon Nikodym   derivative given by
${p_{T-t}( y_0, \sigma_t) \over p_T(x_0,y_0)}.$ 
Define
  $$  \Cyl_t=\{F| F(\sigma)=f(\sigma_{s_1}, \dots, \sigma_{s_k}), f\in C^\infty_K (M^k), 0<s_1<\dots <s_k\leqslant t< T\}. $$ For $F\in \Cyl_t$,
   \begin{eqnarray*}
&&p_T(x_0,y_0)\int_{\C_{x_0}M} f(\sigma_{s_1}, \dots, \sigma_{s_n}) d\mu_{x_0,x_1}(\sigma)\\
& = &\int_{M^n} f(x_1,\dots, x_n)  p_{s_1}(x_0, x_1)\dots 
p_{s_{n}-s_{n-1}}(x_{n-1},x_n)p_{T-s_n}(x_n,y_0)\Pi_{i=1}^n dx_i.
\end{eqnarray*}
 \medskip
 
That this defines a measure on $C_{x_0}M$ due to Kolmogorov's theorem and the assumption that for $\beta>0,\delta>0$,
\begin{equation}
\int \int  d (y, z) ^\beta { p_s(x_0,y)p_{t-s}(y,z)p_{T-t}(z, y_0)\over P_T(x_0,y_0)}  dy dz
\leqslant C |t-s|^{1+\delta},
\end{equation}
whose validity we discuss later.
The Brownian bridge measure $\mu_{x_0, y_0}$ starting at $x_0$ and  ending at  $y_0$ charges only the
subspace, $\C_{x_0, y_0}(M)$. If $x_0=y_0$ the Brownian bridge measure  only charges   the loop space $L_{x_0}M $.

 \subsection{Where is the Problem?}
 
To see where the problem lies we look at the stochastic differential equation representation for the
Brownian bridge measure. The fundamental difference between the dynamic representation for Brownian bridge measure and that  for the Brownian motion measure is that the SDE  for the Brownian bridge is no longer homogeneous and  a singularity develops as  $t$ approaches the terminal time.  The conditioned Brownian motion realisation of the Brownian bridge on the other hand poses a more artificial problem: the conditioned process  is not adapted to the original filtration $\F_t$ of the Brownian motion we started with. It is however adapted to the enlarged filtration  $\G_t=\F_t\vee \sigma\{B_T\}$.

Let  $X: M\times \R^n\to TM$ be a smooth map with $X(x):\R^m\to T_xM$ linear for each $x\in M$ and an isometric surjection.  We assume that for $v\in T_xM$ and $U\in \Gamma TM$,
$$\nabla _vU=L_{Z^v}U(x)$$
where $Z^v(y)=X(y)Y(x)v$. That such a map $X$ exists and defines the given connection was discussed in
\cite{Elworthy-LeJan-Li-book}. Consider the following stochastic differential equation:
\begin{equation}
\label{sde}
dy_t=X(y_t)\circ dB_t.
\end{equation}
Its infinitesimal generator is given by ${1\over 2}\Delta$ for $\Delta$ the Laplacian and the solution is the Brownian motion on $M$. The SDE perturbation by the gradient of the logarithm of the heat kernel 
\begin{equation}
\label{SDE-bridge}
dy_t=X(y_t)\circ dB_t+\nabla \log p_{T-t}(y_t, y_0)dt
\end{equation}
defines a process $(y_t, t<T)$. Here $\nabla$ denotes the Levi-Civita connection.  If \begin{equation}
\int_0^T |\nabla \log p_{T-t}(y_t, y_0)|dt<\infty,
\end{equation}  $\lim_{t\to T}y_t$ is well defined.

On $\R^n$, the time dependent vector field is $-{y_t-y_0\over T-t}$ and exert a strong pull 
on the Brownian particle toward $y_0$.  As the Brownian motion measure and the Brownian bridge measure are equivalent  on $\F_t$ for $t<T$, the Brownian Bridge cannot explode before the terminal time. That the solution gives rise to the measure $\nu_{0,0}$  on the path space restricted to $\F_t, t<T$ is the consequence of the Girsanov transform:  For $t<T$, the law of $\{y_s: s<t\}$
is absolutely continuous with respect to that of the Brownian motion with Radon-Nikodym derivative $N_t$
on $\F_t$  given by
$$e^{\int_0^t  \langle \nabla \log P_{T-s}(x_s), X^*(x_s) dB_s\rangle -{1\over 2}\int_0^t  |\nabla \log P_{T-s}(x_s)|^2 ds}={P_{T-t}(x_t, y_0)\over P_T(x_0,y)}.$$
Hence they agree on cylindrical functions. To show that they agree everywhere, we  only need to
show that $y_t$ has continuous sample path, i.e. for some $p>0$, $\delta>0$,
\begin{equation}
\E d(y_t,y_s)^p\leqslant C |t-s|^{1+\delta}.
\end{equation}

We summarise now all conditions that we need so far

\begin{assumption}[A.]
\begin{enumerate}
\item $\int _M p_t(x,y)dy=1$.
\item For some constant $p>0$ and $\delta>0$, $$\int d(y_t,y_s)^p d\mu_{x_0,y_0}\leqslant C |t-s|^{1+\delta}$$
\item $$\int_0^T |\nabla \log p_{T-t}(y_t, y_0)|dt<\infty, \qquad a.s.$$
\item $$ \int \int  d (y, z) ^\beta { p_s(x_0,y)p_{t-s}(y,z)p_{T-t}(z, y_0)}   dy dz
\leqslant C |t-s|^{1+\delta}.$$
\end{enumerate}
\end{assumption}

Further gradient estimates on the heat kernel are needed for the validity of integration by parts formulae and Clark-Ocone formulae. See e.g. Driver \cite{Driver94} and Aida \cite{Aida2000}. See also Gong-Ma \cite{Gong-Ma} for an alternative formulation of the Clark-Ocone formula.
We summarize the known heat kernel estimates here. 
\begin{itemize} 
\item For $x $ not in the cut locus of $y$, for small $t$ 
$$P_t(x, y)=(2\pi t)^{-n/2}e^{{-d(x, y)^2\over 2t}}\theta_y(x)^{-1\over 2} (1+o(t))$$
where $\theta_y(x)$ is Ruse's invariant.
For hyperbolic space,
$$\theta_1(x_0)= \left({\sinh r(x_0) \over r(x_0)}\right)^{n-1}.$$

\item On a compact manifold $M$, known estimates on the time dependent vector fields are:
\begin{equation}
\label{heat kernel 2}
|\nabla \log p_{T-t}(x, y_0)|\leqslant  C{d(x,y_0)\over T-t}+{C\over \sqrt{T-t}}, \qquad t\in [0, T)
\end{equation}
\end{itemize} 

For the Hyperbolic space, the above assumption holds.  For example it is shown in Aida \cite{Aida2000} that (\ref{heat kernel 2})  holds on the hyperbolic spaces. He used the iteration formula for heat kernels  for $H^n$, iterated on $n$.

 \section{A weak Logarithmic Sobolev Inequality}

  For any torsion symmetric metric connection $\nabla$ on the path space, whose parallel translation along a path $\sigma$ is denoted by $\paral_\cdot$, there is the tangent sub-space to $T_{\sigma}\C_{x_0}M$
$$H_\sigma=\{ \paral_s k_s: k\in L_0^{2,1}(T_{x_0}M)\},$$
which we call the Bismut tangent space with Hilbert space norm induced from the Cameron Martin space.
The tangent space $T_{x_0}M$ is identified with a copy of $\R^n$ through a chosen linear frame $u_0$. 
Let $\mu$ be a probability measure on $C_{x_0}M$ including measures which concentrates on a subspace e.g. the loop space. When there is no confusion of which measure is used, we denote by the integral of an function $f$ with respect to $\mu$ by $\E f$, its variance $\E(f-\E f)^2$ by $\var(f)$ and its entropy $\E f\log {f \over \E f}$ by $\Ent(f)$.

The differential operator $d$ is cloasable whenever Driver's integration by parts formula holds. We define $\D^{1,2}\equiv \D^{1,2}(C_{x_0}M)$ to be the closure of smooth cylindrical function $\Cyl_t$, $t<T$ under this graph norm:
$$\sqrt{\int _{C_{x_0}M} |\nabla f|_{H_\sigma}^2(\sigma) \mu(d\sigma)+\int  f^2(\sigma) d\mu(\sigma)}.$$

\subsection{Aida's inequality and weak Poincar\'e inequalities}
Consider the Laplace Beltrami operator on a complete Riemannian manifold. A Poincar\'e inequality may not hold.  By restriction to an exhausting relatively compact open sets $U_n$, local Poincar\'e inequality always exist. The problem is that the Poincar\'e constant may blow up as $n$ goes to infinity.
In \cite{Eberle-local-gap} Eberle showed that a local Poincar\'e inequality holds for loops spaces over a compact manifold. However the computation was difficult and complicated and there wasn't an estimate on the blowing up rate, although it is promising to obtain a concrete estimate from Eberle's frameworks. Once a blowing up rate for local Poincar\'e inequalities are obtained, we have the so called weak Poincar\'e inequality and in the case of Entropy we have the weak Logarithmic Sobolev inequality.
\begin{eqnarray*}
\var(f) &\leqslant &\alpha (s)\int |\nabla f|^2 d\mu + s |f|_\infty^2,\\
\Ent(f^2)&\leqslant& \beta (s)\int |\nabla f|^2 d\mu + s |f|_\infty^2.
\end{eqnarray*}
We assume that $\alpha$ and $\beta$ to be non-decreasing functions from $(0,\infty)$ to $\R_+$.
These inequalities were studied by Aida \cite{Aida98}, R\"ockner-Wang \cite{Rockner-Wang01},  Barthe-Cattiaux-Roberto \cite{Barthe-Cattiaux-Roberto05}, Cattiaux-Gentil-Guillin \cite{Cattiaux-Gentil-Guillin07}.  
The rate of convergence to equilibrium for the dynamics associated to the Dirichlet form $\int |\nabla f |^2 d\mu$ is strongly linked to Poincar\'e inequalities. See Aida-Masuda-Shigekawa \cite{ Aida-Masuda-Shigekawa}, Aida \cite{Aida98},  Mathieu \cite{Mathieu98}, and R\"ockner-Wang \cite{Rockner-Wang01}. In the case of weak Poincar\'e inequalities,  exponential convergence is no longer guaranteed. Also the weak Poincar\'e inequality holds for any $\alpha$  is equivalent to Kusuoka-Aida's weak spectral gap inequality which states that any mean zero sequence of functions $f_n$ in $\D^{1,2}$ with $\var(f_n)\leqslant 1$ and $\E( |\nabla f|^2 )\to 0$ is a sequence which converges to $0$ in probability.

  \begin{proposition}
  \label{wlsiH}
  Let $\mu$ be any probability measure on $\C_{x_0}M$ with the property that there exists a positive 
  function 
  $u \in \D^{1,2}$ such that Aida's type inequality holds:
    \begin{equation}
  \label{wlsi5}
  \Ent(f^2) \leqslant \int u^2  |\nabla f|^2  d\mu, \qquad \forall f\in  \D^{1,2}\cap L_\infty
  \end{equation}
Assume furthermore that  $|\nabla u|\leqslant a$   and $\int e^{Cu^2} d\mu<\infty$ for some $C,a>0$.
 Then for all functions $f$ in $ \D^{1,2}\cap L_\infty$
 \begin{equation}\label{wlsi6}
\Ent(f^2)\leqslant \beta (s)\int |\nabla f|^2 d\mu + s |f|_\infty^2,
 \end{equation}
 where $\beta(s)=C| \log s|$ for $s<r_0$ where $C$ and $r_0$ are constants. 
\end{proposition}

\begin{proof}

 Let $\a_n: \R\to [0,1]$ be a sequence of smooth functions 
approximating $1$ such that  
\begin{equation}
\a_n(t)=\left\{ \begin{array}{ll}
1  \qquad   &t\leqslant n-1\\
\in [0,1],  \qquad & t \in (n-1,n)\\
0& t \geqslant  n
\end{array}\right.
\end{equation}
We may assume that $|\a_n'|\leqslant 2$. Define $$f_{n}= \a_n(u)f$$
for $u$ as in the assumption. 
 Then $f_{n}$ belongs to $ \D^{1,2}\cap L_\infty $ if $f$ does. We may apply Aida's  inequality (\ref{wlsi5})
 to $f_n$.  The gradient of $f_n$ splits into two parts of which one involves  $f$ and the other involves $\nabla f$. The part involving the gradient vanishes outside of the region of $A_n:=\{\omega: u(\omega)<n\}$ and on  $A_n$  it is controlled by $g$ and therefore by $n$. The part involving $f$ itself vanishes outside 
 $\{\omega: n-1<u(\omega)<n\}$ and the probability of  $\{\omega: n-1<u(\omega)<n\}$ is very small by the exponential integrability of $u$.  We split the entropy into two terms:
 $\ent(f^2)=\ent(f_n^2)+[\ent(f^2)-\ent(f_n^2)], $ to the first we apply the Sobolev inequality (\ref{wlsi5}).  \begin{eqnarray*}
 \int f_{n}^2 \log {f_{n}^2 \over \E f_{n}^2} d\mu
 &\leqslant&\int u^2 |\nabla f_{n}|^2 d\mu\\
 &\leqslant&\int u ^2 \left[  |\nabla f|\a_n(u) + |\a_n'| |\nabla u| f\right]^2d\mu\\
 &\leqslant & \int_{u<n} 2 u^2  |\nabla f|^2 \a_n^2(u) d\mu 
 +4a^2\int _{n-1<u<n} u^2f^2  d\mu\\
&\leqslant&2n^2 \int  |\nabla f|^2 d\mu + 4a^2n^2|f|^2_\infty \; \mu(n>u>n-1).
\end{eqnarray*}

Next we compute the difference between $\ent(f^2)$ and $\ent(f^2_{n})$.  
\begin{eqnarray*}
&&\ent(f^2)-\ent(f_{n}^2)=
\int \left(f^2 \log {f^2\over \E f^2}-f_{n}^2 \log {f_{n}^2 \over \E f_{n}^2} \right)d\mu \\
&&=\int \left( 1- \a_n^2(u) \right)  f^2\log {f^2\over \E f^2}d\mu
+ \int f^2\a_n^2(u)   \left(  \log {f^2\over \E f^2}- \log { \a_n^2(u)f^2 \over \E  \a_n^2(u)f^2} \right)d\mu\\
&=&I+II.
\end{eqnarray*}

Observe that 
\begin{eqnarray*}
I=&& \int \left( 1- \a_n^2(u) \right)  f^2\log {f^2\over \E f^2}\, d\mu\\
&&\leqslant \int _{u>n-1} f^2( 1- \a_n^2(u)) \log {f^2\over \E f^2}d\mu\\
&&\leqslant 2 |f|^2_\infty\; \int _{u>n-1}\log {|f|\over \sqrt{\E f^2}}\;d\mu.
\end{eqnarray*}
By the elementary inequality $\log x\leqslant x$ and Cauchy-Schwartz inequality
\begin{eqnarray*}
&&I\leqslant2 |f|^2_\infty  \int_{u>n-1} \left({|f|\over \sqrt{\E f^2}}\right)\;d\mu\\
&&\leqslant 2 |f|^2_\infty \sqrt{\E \left({|f|\over \sqrt{\E f^2}}\right)^2} \sqrt{\mu (\{u>n-1\})}\\
&&\leqslant 2|f|^2_\infty \sqrt{\mu(u>n-1)}.
\end{eqnarray*}
For the second term of the sum,  with the convention that $0\log 0=0$,
$$II=- \int_{n-1<u<n}  f^2\a_n^2(u)\log \a_n^2(u) d\mu
+\int_{u<n}  f^2\a_n^2(u)  \log {\E f^2\alpha_n^2(u)\over \E f^2}\; d\mu$$
Using the fact that  $ \log {\E f^2\alpha_n^2(u)\over \E f^2}\leqslant 0$ from  $\alpha_n^2(u)\leqslant 1$
and $x\log x \geqslant  -{1\over e}$, we see that
\begin{eqnarray*}
II&\leqslant &{1\over e} \int_{n-1<u<n}  f^2 d\mu
\le{1\over e}\;(|f|_\infty) ^2\cdot  \; \mu({n-1<u<n}).
\end{eqnarray*}
 Finally  adding the three terms together to obtain
 \begin{eqnarray*}
 \int f^2 \log {f^2\over \E f^2} d\mu
&&\leqslant 2n^2 \int  |\nabla f|^2 d\mu  +\big(4a^2n^2+{1\over e}\big)|f|^2_\infty\; \mu(n-1<u<n)\\
&&\quad + |f|^2_\infty \sqrt{\mu(u>n-1)}
\end{eqnarray*}
which can be further simplified to the following estimate:
\begin{equation}
\label{wsi6}
 \int f^2 \log {f^2\over \E f^2} \, d\mu
 \leqslant 2n^2 \int  |\nabla f|^2 d\mu  +\big(4a^2n^2+{1\over e}+1\big)|f|^2_\infty \sqrt{\mu(u>n-1)}.
\end{equation}
The  exponential integrability of $u$ will supply the required estimate on the tail probability,
$$\sqrt{\mu(u>n-1)}\leqslant  e^{-{C\over 2}(n-1)^2} \sqrt{\E e ^{Cu^2}}$$
Define   $b(r)=( 4a^2r^2+{1\over e}+1)e^{-{C\over 2}(r-1)^2}$. Then
$$ \int f^2 \log {f^2\over \E f^2} d\mu\leqslant 2n^2 \int  |\nabla f|^2 d\mu  +b(n)|f|^2_\infty.$$
For $r$ sufficiently large, $b(r)$ is a strictly monotone function  whose  inverse function is denoted by $b^{-1}(s)$ which decreases exponentially fast to $0$. Define
$\beta(s)=b^{-1}(2s^2)$.  For any $s$ small choose $n(s)$ to be the smallest integer such that  $s \geqslant  b(n)$. Then
$$ \int f^2 \log {f^2\over \E f^2} d\mu\leqslant \beta(s) \int  |\nabla f|^2 d\mu +s|f|^2_\infty $$
 Here $\beta(s)$ is of  order $|\log s|$ as $s\to 0$.
 \end{proof}

Note that in the above proof we only needed the weak integrability of the function $u^2$, or
the estimate $\mu(u>n-1)$. This leads to the following :
\begin{remark}
If (\ref{wlsi5}) holds for $u\in\D^{2,1}$ with the property $|\bigtriangledown u|\leqslant a,u \geqslant 0$ and
$$\mu(u^2>s^2)<m^2(s),$$
for a non-increasing function $m$ of the order $o(s^{-2})$,
then by (\ref{wsi6}), the weak Poincar\'e inequality holds with $\beta(s)$ of the order of the inverse function of $(s^2+2)m(s)$.
\end{remark}

\subsection{Relation between various inequalities}
The functional inequalities for a measure describes how the $L^2$ or other norms of a function is controlled by its derivatives with a universal constant.  They describe the concentration of an admissible function
around its mean. 
A well chosen gradient operator is used to give these control. On the other hand concentration inequalities are related intimately with isoperimetric inequalities. For finite dimensional spaces it was shown in the remarkable works of Cattiaux-Gentil-Guillin \cite{Cattiaux-Gentil-Guillin07} and   Barthe-Cattiaux-Roberto \cite{Barthe-Cattiaux-Roberto05} for measures in  finite dimensional spaces  one can pass from capacity  type of inequalities to weak Logarithmic Sobolev inequalities  and vice versa with great precision. Similar results holds for
weak Poincar\'e inequalities. This gives a great passage between the two inequalities. 
We give here a direct proof that this works wonderfully in infinite dimensional spaces. The proof is somewhat standard and is inspired by the  two previous mentioned articles and that of  Ledoux \cite{Ledoux-inequality-94}.

\begin{proposition}
\label{pr:entropy-variance}
If for all $f$ bounded measurable functions  in $\D^{2,1}(C_{x_0}M)$,   the weak logarithmic Sobolev inequality holds for $0<s<r_0$, some given $r_0>0$,
$$\Ent(f^2)   \leqslant \beta(s) \int |\nabla f|^2 d\mu+s|f|^2_\infty$$
where $\beta(s)=C\log{1\over s}$ for some constant $C>0$,
Then Poincar\'e inequality
$$\var(f)   \leqslant \alpha\int |\nabla f|^2 d\mu.$$
holds for some constant $\alpha>0$.
\end{proposition}

\begin{proof}
By the minimizing property of the variance for any real number $m$, 
\begin{equation}
\label{min-var}
\var(f)\leqslant \int  ((f-m)^+)^2d\mu+\int  ((f-m)^-)^2 d\mu.
\end{equation}
We choose $m$ to be the median of $f$ such that $\mu(f-m>0)\leqslant {1\over 2} $ and $\mu(f-m<0)\leqslant {1\over 2}$. 

Let $g$ be a positive  function in $\D^{2,1}$ such that  $\int g^2 d \mu=1$ and $\mu\{g\not = 0\}\leqslant {1\over 2}$. Here we  take $g=g_1$ or $g=g_2$ for
\begin{equation}
\label{10}
g_1={(f-m)^+ \over  \sqrt{\int  ((f-m)^+)^2d\mu}\quad}  \hbox { or }  
\quad g_2={(f-m)^-\over  \sqrt{\int  ((f-m)^-)^2d\mu}}.
\end{equation}
For  $ \delta_0>0$ and $\delta>1$ and $0<\delta_0<\delta_1<\delta_2<\dots $ with $\delta_n=\delta_0\delta^n$,
\begin{equation*}\begin{split}
\E g^2 &=\int_0^{+\infty}2s\mu(|g|>s)ds\\
&=\int_0^{\delta_1}2s \;\mu(|g|>s)ds+\sum_{n=1}^{\infty}
\int_{\delta_{n}}^{\delta_{n+1}}  2s\mu(|g|>s)ds\\
&\leqslant \int_0^{\infty}2s\;  \mu(|g|\wedge \delta_1>s)ds +\sum_{n=1}^{\infty}
\int_{\delta_{n}}^{\delta_{n+1}}  2s\mu(|g|>s)ds\\
\end{split}
\end{equation*}
Consequently we have,
\begin{equation}
\label{g-square}
\E g^2 \leqslant \E(g\wedge \delta_1)^2 +\sum_{n=1}^{\infty}
\int_{\delta_{n}}^{\delta_{n+1}}  2s\mu(|g|>s)ds 
\end{equation}
Define  $$I_1:=\E(g\wedge \delta_1)^2, \qquad I_2:=\sum_{n=1}^{\infty}
\int_{\delta_{n}}^{\delta_{n+1}}  2s\mu(|g|>s)ds.$$

Recall the following entropy inequality. If $\varphi: \Omega\to [-\infty, \infty)$ is a function such that
$\E e^{\varphi}\leqslant 1$ and  $G$ is a real valued random function such that $\varphi$ is finite on the support of $G$, then 
$$\int G^2 \varphi d\mu\leqslant \ent(G^2).$$
Here we take the convention that $G^2\varphi=0$ where $G^2=0$ and $\varphi=\infty$.
Let
\begin{equation*}
\varphi:=\begin{cases}
\ \log2 &\text{if}\  g>0, \\
-\infty &\text{otherwise.}
\end{cases}  
\end{equation*}a
then $\int e^{\varphi}d\mu=2\mu(g\neq 0)\leqslant 1$. Hence 

$$\ent((g\wedge\delta_1)^2) \geqslant  
\int(g \wedge\delta_1)^2 \varphi d\mu$$
so that  $$\E((g\wedge\delta_1)^2)\leqslant 
\frac{1}{\log2}\ent((g\wedge\delta_1)^2). $$
We apply the  weak logarithmic Sobolev inequality 
$$\Ent(f^2)\leqslant \beta(r) \E |\nabla f|^2 +r|f|_\infty^2$$
to  $g\wedge \delta_1$ to obtain, for some $r<r_0$,
\begin{equation} \label{g-square-part-1}
\E(g\wedge\delta_1)^2 
\leqslant \frac{\beta(r)}{\log2}\cdot 
\int |\nabla g|^2 1_{g<\delta_1}\; d\mu+ \frac{r\cdot \delta_1^2}{\log 2}.
\end{equation}


Now we are going to estimate $I_2$. For $n=0,1,\dots$, let  
$$g_n=   (g-\delta_n)^+\wedge (\delta_{n+1}-\delta_n) .$$
Then $g_n \in \D^{1,2}$, $\E g_n^2 \leqslant 1$ and $$|\nabla g_n|\leqslant  
|\nabla g|1_{\delta_n\leqslant g\leqslant \delta_{n+1}}.$$
From $g_n  \geqslant (\delta_{n+1}-\delta_n)   I_{ \{ g>\delta_{n+1} \}}$,
$$\mu(g>\delta_{n+1})\leqslant    { \E g_n^2 \over  (\delta_{n+1}-\delta_n)^2}.$$
Next we observe that for $n \geqslant 1$,

\begin{equation}
\label{estimate-15}
\begin{split}
\int_{\delta_{n}}^{\delta_{n+1}}  2s\; \mu(|g|>s)ds
& \leqslant\mu(g>\delta_n)\cdot(\delta_{n+1}^2-\delta_n^2)  \\
& \leqslant \frac{\delta_{n+1}^2-\delta_n^2}{(\delta_{n}-\delta_{n-1})^2}\E g_{n-1}^2 \\
& =\delta^2{\delta+1\over \delta-1} \E g_{n-1}^2.
\end{split} 
\end{equation}




Next  we compute $\E g_n^2$.  We'll chose a function $\varphi_n$ which can be used to estimate the $L^1$ norm of $g_n^2$ by its entropy. Define
 \begin{equation*}
 \label{estimate-16}
\varphi_n:=
\begin{cases}
\log \delta_n^2 &\text{if}\  g>\delta_n, \\
-\infty &\text{otherwise.}
\end{cases} 
\end{equation*}
Then $\int e^{\varphi_n} d\mu=\delta_{n}^2\mu(g>\delta_n)\leqslant 1$, hence
$$\ent(g_n^2) \geqslant  \int g_n^2\varphi_n d\mu.$$
Thus, 
\begin{equation}
\E g_n^2\leqslant \frac{1}{\log\delta_n^2}\ent(g_n^2)
\leqslant \frac{1}{2\log \delta_0+ 2n\log\delta}\ent(g_n^2) .
\end{equation}
By (\ref{estimate-15}) and  (\ref{estimate-16})  the second term in $\E g^2$ is controlled by the entropy of the functions $g_n^2$ to which we may apply  the weak logarithmic Sobolev inequality with constants $r_n<r_0$.  The constant $r_n$ are to be chosen later.
\begin{equation}
\label{estimate-17}
\begin{split}
&\int_{\delta_{n+1}}^{\delta_{n+2}}  2s\; \mu(|g|>s)ds\\
&\leqslant \delta^2{\delta+1\over \delta-1}\cdot
\frac{1}{2\log \delta_0+ 2n\log\delta}\ent(g_{n}^2)\\
&\leqslant \delta^2{\delta+1\over \delta-1}\cdot
\frac{1}{2\log \delta_0+ 2n\log\delta}
\bigg(\beta(r_n)\int |\nabla g|^2 I_{\delta_{n}\leqslant g <\delta_{n+1}} d\mu
+ r_n\cdot |g_{n}|_{\infty}^2\bigg).
\end{split}
\end{equation}
Note that $|g_n|_{\infty}\leqslant \delta_{n+1}-\delta_n $ and  summing up in $n$ we have,
\begin{equation}
\label{estimate-18}
\begin{split}
I_2& \leqslant \frac{\delta^2(\delta+1)}{2(\delta-1)}\sum_{n=0}^{\infty}
\frac{\beta(r_n)}{\log \delta_0+n \log \delta}\int |\nabla g|^2 
I_{\delta_{n}\leqslant g <\delta_{n+1}} d\mu\\
&+\frac{\delta^2-1}{2}\sum_{n=0}^{\infty}
\frac{\delta_0^2\cdot \delta^{2n+2}}{\log \delta_0+ n\log\delta}\cdot r_n 
\end{split}
\end{equation}
Denote 
$$b_{-1}={\beta(r)\over \log 2}, \quad b_n=\frac{\delta^2(\delta+1)}{2(\delta-1)}
\frac{\beta(r_n)}{\log \delta_0+n \log \delta}$$ and
$$c_{-1}= \frac{r\cdot \delta_1^2}{\log 2}, \quad c_n=\frac{\delta^2-1}{2}\sum_{n=0}^{\infty}
\frac{\delta_0^2\cdot \delta^{2n+2}}{\log \delta_0+ n\log\delta}\cdot r_n $$
 Finally combining (\ref{g-square-part-1})  with (\ref{estimate-18}) we have
$$\E g^2\leqslant  \sum_{n=-1}^\infty b_n \int |\nabla g|^2 1_{\{ \delta_{n-1}\leqslant g <\delta_{n}\}}d\mu
+\sum_{n=-1}^\infty c_n$$
 We'll next choose $r_n$ so that $\sum c_n<1/2$ and that the sequence $b_n$ has an upper bound. This is fairly easy by choosing that $r_n$ of the order ${\delta^{-(2n+2)}\over n}$. Taking $g=g_1$, we see that
 \begin{eqnarray*}
1= \E g_1^2&\leqslant&\sup_n (b_n)\int |\nabla g_1|^2 d\mu +\sum c_n\\
& \leqslant & \sup_n (b_n){1\over \E[(f-m)^+]^2}\int |\nabla f|^2 1_{\{f>m\}} d\mu +\sum c_n.
 \end{eqnarray*}
Hence $$ \E[(f-m)^+]^2\le 2\sup_n (b_n) \int |\nabla f|^2 1_{\{f>m\}} d\mu,$$
 $$ \E[(f-m)^-]^2\le 2\sup_n (b_n) \int |\nabla f|^2 1_{\{f>m\}} d\mu.$$
The Poincar\'e inequality follows.
\end{proof}

\begin{remark} We could optimize the constant in the Poincar\'e inequality.  For example  when $r_0=1/2$, we let $\epsilon=1/8, \delta=\sqrt{2}, \delta_0=2^{\frac{9}{2}}$, the Poincar\'e constant is  approximately  $40.82C$, which is smaller than that given  in  Cattiaux-Gentil-Guillin \cite{Cattiaux-Gentil-Guillin07}.
However we do not expect to have a sharp estimate on the constant.
\end{remark}
\begin{proof}
We need to choose the $r_n$, $\delta$, $\delta_0$ carefully 
 to optimise on the constant. Assume that $\E g^2=1$ for simplicity. We choose suitable constants $\delta_0,\ \delta,\ \epsilon$ satisfying
$\frac{\epsilon}{\delta_0^2\cdot\delta^2}<r_0$ and take $r=\frac{\epsilon}{\delta_0^2\cdot\delta^2}$
in $b_{-1}$ and recall that $\beta(s)=C\log\frac{1}{s}$ here. Then
\begin{equation}
\label{estimate-19}
I_1=\E(g\wedge\delta_1)^2 
\leqslant \frac{C\cdot\log\big(\frac{\delta_0^2\cdot \delta^2}{\epsilon}\big)}{\log2}\cdot 
\int |\nabla g|^2 1_{g<\delta_1}\; d\mu+ \frac{\epsilon}{\log 2}
\end{equation}
Next we take 
$$r_n:=\frac{\log \delta^2 }{\delta_0^2\cdot\delta^{2n+2}(\log\delta_0 + n\log\delta )^2}=   \frac{1}{\delta_0^2\cdot\delta^{2n+2}}\cdot
\frac{1}{({\log\delta_0\over \log\delta} +n)^2}.$$
Choose $\delta_0,\delta$ so that  $r_n < r_0$ for each $n \geqslant 0$, 
and $\log\log \delta_0>0$. For simplicity we denote 
$A:=\frac{\log\delta_0}{\log \delta}$.  Note that 
$\beta(s)=C\log\frac{1}{s}$ in (11) and  
$$\log\frac{1}{r_n}=2\log\delta_0+2(n+1)\log\delta+
2\log(A+n)$$
It follows that

\begin{equation}
\label{estimate-20}
\begin{split}
I_2 &\leqslant C\delta^2\cdot{\delta+1\over \delta -1}
\sum_{n=0}^{\infty}\big(1+\frac{1}{n+A}
+\frac{\log(n+A)}{n+A}
\cdot\frac{1}{\log\delta}\big)\int |\nabla g|^2 I_{\{\delta_n\leqslant g <\delta_{n+1}\}} d\mu\\
& +\frac{\delta^2-1}{2\log\delta}\sum_{n=0}^{\infty} 
\frac{1}{(n+A)^3}\\
&\leqslant C\delta^2\cdot{\delta+1\over \delta -1}\big(1+
\frac{1}{A}+\frac{\log A}{A}\cdot\frac{1}{\log\delta}\big)\int |\nabla g|^2 I_{\{\delta_0\leqslant g\}} d\mu
+\frac{\delta^2-1}{4\log\delta}\cdot\frac{1}{(A-1)^2}
\end{split}
\end{equation}

Let $$C_1(\delta, \delta_0, \epsilon):=C\delta^2\cdot{\delta+1\over \delta -1}\big(1+
\frac{1}{A}+\frac{\log A}{ A}\cdot\frac{1}{\log\delta}\big)$$
$$C_2(\delta, \delta_0, \epsilon):=
\frac{C\cdot\log\big(\frac{\delta_0^2\cdot \delta^2}{\epsilon}\big)}{\log2}$$
$$C_3(\delta, \delta_0, \epsilon):=\frac{\delta^2-1}{4\log\delta}\cdot
\frac{1}{(A-1)^2}
+\frac{\epsilon}{\log2} $$

So from (\ref{estimate-19}) and (\ref{estimate-20}) and the assumption $\E g^2=1$, we have:
\begin{equation}
\E g^2\leqslant \frac{C_1(\delta, \delta_0, \epsilon)+C_2(\delta, \delta_0, \epsilon)}
{1-C_3(\delta, \delta_0, \epsilon)}\int |\nabla g|^2 d\mu  
\end{equation}
provided we choose suitable constants $\delta, \delta_0, \epsilon$
to make $C_3(\delta, \delta_0, \epsilon)<1$. Apply the above estimate to $g_1$, $g_2$ and these together with
(\ref{min-var}) give the required inequality.

\end{proof}

When the function $\beta(s)$ in weak logarithmic Sobolev inequality is of order greater than 
$\log \frac{1}{s}$, we no longer have a Poincar\'e inequality, but a weak Poincar\'e inequality 
is expected.  In fact there is the following relation. The finite dimensional version can be found in \cite{Barthe-Cattiaux-Roberto05}). We give here a direct proof without going through any capacity type inequalities.
\begin{remark}
If for all bounded measurable functions  $f$ in $\D^{1,2}(C_{x_0}M)$,   the weak logarithmic Sobolev inequality holds for $s< r_0$, some given $r_0>0$ and a non-increasing function $\beta: (0,r_0)\longmapsto R^+$,
$$\ent(f^2)  \leqslant \beta(s)  \E |\nabla f|^2 +s|f|^2_\infty$$
Then there exist constants $r_1>0, C_1, C_2$ such that for all $s<r_1$, the weak Poincar\'e inequality
$$\var(f)   \leqslant \frac{\beta\left (C_2s\log\frac{1}{s}\right)}
{C_1\log\frac{1}{s}}\E (|\nabla f|^2)+s|f|^2_\infty.$$
holds.

\end{remark}

\begin{proof} As a Poincar\'e inequality is not expected, we need to cut off the integrand at infinity. 
We keep the notation of the proof of Proposition \ref{pr:entropy-variance}. Let  $\delta_n=\delta_0\cdot \delta^n$ for some $\delta_0>1,\ \delta>1$ and the function $g$  as in (\ref{10}). We have
\begin{equation}
\label{estimate-21}
\begin{split}
&\E g^2=\E(g\wedge \delta_1)^2+
\sum_{n=1}^{N+1}\int_{\delta_n}^{\delta_{n+1}}2s\mu(g>s)ds\\
&+\sum_{n=N+1}^{2N+1}\int_{\delta_n}^{\delta_{n+1}}2s\mu(g>s)ds
+\int_{\delta_{2N+2}}^{\infty}2s\mu(g>s)ds
\end{split} 
\end{equation}
First from $\E g^2=1$, we have the following tail behaviour:
\begin{equation}
\label{estimate-26} 
\int_{\delta_{2N+2}}^{\infty}2s\mu(g>s)ds=\E(g^2-\delta_{2N+2})_{+}^2
\leqslant |g|_{\infty}^2\mu(g>\delta_{2N+2})\leqslant 
\frac{1}{\delta_0^2\delta^{4N+4}}|g|_{\infty}^2
\end{equation} 
We now consider $\delta^{4N+4}$ to be of order $1/s$. 
For the first two terms of (\ref{estimate-21}), we use estimates from the previous proof.
First recall  (\ref{g-square-part-1}), \begin{equation*}
\E(g\wedge\delta_1)^2 
\leqslant \frac{\beta(r)}{\log2}\cdot 
\int |\nabla g|^2 1_{g<\delta_1}\; d\mu+ \frac{r\cdot \delta_1^2}{\log 2}
\end{equation*}

Next by (\ref{estimate-17}), we have:
\begin{equation*}
\begin{split}
&\sum_{n=1}^{N+1}\int_{\delta_n}^{\delta_{n+1}}2s\mu(g>s)ds\\
&\leqslant C_2 \sum_{n=0}^{N}
\frac{\beta(r_n)}{n+C_3 }\int |\nabla g|^2 
I_{\delta_n\leqslant g <\delta_{n+1}} d\mu
+C_2\sum_{n=0}^{\infty}
r_n\cdot\frac{\delta^{2n}}{n+C_3}
\end{split} 
\end{equation*}
Here $C_2, C_3$ are some constants depending on $\delta_0$ and $\delta$ and $C_3={\log \delta_0\over \log \delta}$. 
For $n=0,1,\dots, N$, take $$r_n= \frac{1}{\delta^{2n}\cdot(n+C_3)}.$$   
We may assume that  $\beta(r)$ is an increasing function of order greater than $\log( {1\over r})$ for $r$ small, in which case $$\frac{\beta(\frac{1}{\delta^{2n}\cdot(n+C_3)})}{n+C_3}$$
is an increasing function of $n$ for $n$ sufficiently large. Hence
\begin{equation}
\label{estimate-23}
\sum_{n=1}^{N+1}\int_{\delta_n}^{\delta_{n+1}}2s\mu(g>s)ds
\leqslant C_2\frac{\beta\left(\frac{1}{\delta^{2N}\cdot(N+C_3)}\right)}{N+C_3}
\int |\nabla g|^2 
I_{\delta_0\leqslant g <\delta_{N+1}} d\mu+\frac{C_2}{C_3-1}.
\end{equation}

If we apply this estimate to the whole range $n\le 2N$, $\beta(r_{2N})\over 2N$ would be the order of 
$\beta({s\over |\log s|})$. However to make the estimate more precise, we take a different rate function $r_n$ for $N+1\le n\le 2N$. Let  $r_n=\frac{N}{\delta^{4N}}$ in (\ref{estimate-24}) and we will give a more precise estimate on $|g_n|_\infty$.  Apply (\ref{estimate-17}) again to the sum from $N+1$ to $2N$ in (\ref{estimate-21}) \begin{equation}
\label{estimate-24}
\begin{split}
&\sum_{n=N+1}^{2N+1}\int_{\delta_n}^{\delta_{n+1}}2s\mu(g>s)ds\\
&\leqslant C_2 \sum_{n=N}^{2N}
\frac{\beta(r_n)}{n+C_3 }\int |\nabla g|^2 
I_{\delta_n\leqslant g <\delta_{n+1}} d\mu
+C_2\sum_{n=N}^{2N}
\frac{r_n}{n+C_3}\cdot  |g_n|_{\infty}^2.
\end{split} 
\end{equation}

Since $g$ is bounded, there is $k$ such that  $\delta_k<|g|_{\infty}\leqslant \delta_{k+1}$ for some 
integer $k$.
\begin{equation*}
\begin{split}
\sum_{n=0}^{\infty}|g_n|_{\infty}^2
&= \sum_{n=0}^{k-1}(\delta_{n+1}-\delta_n)^2+(|g|_{\infty}-\delta_{k})^2\\
&\leqslant \bigg(\sum_{n=0}^{k-1}(\delta_{n+1}-\delta_n)+
|g|_{\infty}-\delta_{k}\bigg)^2\\
&= (|g|_{\infty} -\delta_0)^2
\end{split}
\end{equation*}
Hence
$$\sum_{n=0}^{\infty}|g_n|_{\infty}^2\leqslant |g|_{\infty}^2.   $$
Recall that $r_n={N\over \delta^{4N}}$,
\begin{equation}
\label{estimate-25}
\begin{split}
&\sum_{n=N+1}^{2N+1}\int_{\delta_n}^{\delta_{n+1}}2s\mu(g>s)ds\\
&\leqslant C_2\cdot\frac{\beta(\frac{N}{\delta^{4N}})}{N+C_3}
\int |\nabla g|^2 
I_{\delta_{N}\leqslant g <\delta_{2N+1}} d\mu+\frac{C_2}{\delta^{4N}}|g|_{\infty}^2
\end{split}
\end{equation}
Now adding estimates to all terms in (\ref{estimate-21}) together,   (\ref{estimate-26}-\ref{estimate-25}),  and rearrange the constants. We also note that $\E g^2=1$ and obtain
for $N$ large enough
\begin{equation}
\label{estimate-27}
1\leqslant C_1\frac{\beta(\frac{N}{\delta^{4N}})}{N}\int |\nabla g|^2 d\mu+
\frac{C_2}{\delta^{4N}}|g|_{\infty}^2 +{r \delta_1^2 \over \log 2}+{C_2\over C_3-1}
\end{equation}
Here we use the monotonicity of $\beta$:
$\beta(\frac{N}{\delta^{4N}}) \geqslant \beta(\frac{1}{\delta^{2N}\cdot(N+C_3)}) $. Take $r$ small and $\delta_0$ large so that ${r \delta_1^2 \over \log 2}+{C_2\over C_3-1}<1$. Let $s=\frac{1}{\delta^{4N}}$ in (\ref{estimate-27}), the required result follows.

\end{proof}

\begin{corollary}
 Let $\mu$ be a probability measure. Suppose that there is a positive function 
 $\mu(u^2>s^2)\sim m^2(s)$ some increasing function $m$ of order $o(s^{-2})$ for $s$ small and
 such that $|\nabla u|\leqslant a,a>0$ and for all $f\in \D^{2,1}$, 
\begin{equation}
\label{wsi10}
\Ent(f^2) \leqslant \int g=u^2|\nabla f|^2 d\mu
\end{equation}
Then for $s$ small, 
$$\var(f) \leqslant (r^2|\log r|+\frac{2}{|\log r|})m(r|\log r|)\int|\nabla f|^2a\mu +s |f|^2_\infty.$$
\end{corollary}

\begin{remark}
The results in this section hold for any Hilbert norm on $\HH_\sigma$ including that used in 
Elworthy-Li \cite{Elworthy-Li-ibp}. It also works for a measure on the free path space $CM=\cup_{x_0\in M} C_{x_0}M$  in the following sense. If $\mu_{x_0}$ is a probability measure on $C_{x_0}M$ and $\nu$ a probability measure on $M$, we consider on $CM$  the measure $\mu=\int_M \mu_{x_0} d\nu$.
\end{remark}

\section{Poincar\'e Inequality on Hyperbolic Space}

Aida, \cite{Aida2000}, showed that, for $M$ the standard hyperbolic space, of constant negative curvature. We may assume that the curvature is $-1$. Take the gradient $\nabla$ to be that related to the Levi-Civita connection. 
\begin{equation}\label{Aida}
\int_{C_{x_0}H^n} f^2\log {f^2\over \log |f|^2_{L^2} } d\mu_{x_0,y_0}(\gamma)\leqslant  \int_{C_{x_0}H^n} C(\gamma)|\nabla f|^2 d\mu_{x_0,y_0}(\gamma)
\end{equation}
for $C(\gamma)=C_1(n)+C_2(n)\sup_{0\leqslant t\leqslant 1}d^2(\gamma_t, y_0)$. His method of proof is the Clark-Ocone formula approach. From an integration by parts formula he obtained the following  Clark-Ocone formula by the integration representation theorem:
 $$\E^{\mu_{x_0,y_0}}\{ F|\G_t\}=\E ^{\mu_{x_0,y_0}}F+\int_0^t\langle H_s(\gamma), dW_s\rangle,$$
 where $W_t$ is the anti-development of the Brownian bridge and 
 $$H(s, \gamma)=\E^{\mu_{x_0,y_0}}\{L(\gamma){d\over ds}{\nabla F(\gamma)(s)} | \G_s\}$$
 almost surely with respect to the product measure $dt\otimes \mu_{x_0,y_0}$. Here $\G_t$ is the filtration generated by $\F_t$ and the end point of the Brownian bridge.
 The main obstruction here is that $L$ is random  and careful estimates on $L$ leads to 
 (\ref{Aida}).

\begin{theorem}
Let $M=H^n$, the hyperbolic space of constant curvature $-1$. Then Poincar\'e inequality holds for the Brownian bridge measure $\mu_{x_0, x_0}$. 
\end{theorem}

\begin{proof}
Just note that by  the time reversal of the Brownian bridge and its symmetric property and  the concentration  property of the Brownian motion measure
$$\int_{\C_{x_0}M}  e^ {C d^2(\sigma, y_0) }d\mu_{x_0, y_0} (\sigma)<\infty.$$
Hence by Proposition \ref{wlsi5} and Proposition \ref{pr:entropy-variance} and (\ref{Aida}), we finish the proof.
\end{proof}
\begin{remark}
Aida has shown that inequality (\ref{wlsi5}) holds for the loop space and each homotopy class of  the free loop space over a  compact Riemannian manifold of constant negative curvature. Our discussion earlier shows that Poincar\'e inequality holds in this case.\end{remark}
A compact Riemannian manifold of constant negative curvature is of the form $M=G/H^n$ where $G$ is a discrete subgroup of the isometry group of the hyperbolic space.
The free loop space is the collection of all loops.   See Aida \cite{Aida2000} for precise formulation.

\end{document}